\documentclass[12pt]{article}
\usepackage{mathtools}
\usepackage{amsmath,amssymb,amsbsy}
\usepackage{graphicx}
\usepackage{multirow}

\usepackage{amsfonts}
\usepackage{amsthm}
\usepackage{amssymb}
\usepackage{graphicx, enumerate}
\usepackage{amsmath}
\usepackage{latexsym}
\usepackage{longtable}
\usepackage{tabularx}
\usepackage{amsmath}
\usepackage{amsfonts}
\usepackage{amsthm}
\usepackage{setspace}
\usepackage{graphicx}
\usepackage{float}
\usepackage{rotating}
\usepackage{tikz}
\usepackage{verbatim}
\usepackage{soul}

\def\zet{\mathbb{Z}}

\def\fk2{\lfloor\frac{k}{2}\rfloor}
\def\ck2{\lceil\frac{k}{2}\rceil}
\def\k1{\lfloor\frac{k+1}{2}\rfloor}

\def\Ge {\mathbb{G}}
\def\Ce {\mathbb{C}}
\def\Pe {\mathbb{P}}

\newtheorem{theorem}{Theorem}[section]
\newtheorem{lemma}[theorem]{Lemma}
\newtheorem{conjecture}[theorem]{Conjecture}

\title{On some graph-cordial Abelian groups}
\author{Sylwia Cichacz$^{1,}$\\\\
\normalsize $^1$AGH University of Science and Technology, \vspace{2mm} Poland}

\begin{document}

\maketitle
\begin{abstract}
Hovey introduced $A$-cordial labelings  as a  generalization of cordial and harmonious labelings  \cite{Hovey}. If $A$ is an Abelian group, then a labeling $f \colon V (G) \rightarrow A$ of the vertices of some graph $G$ induces an edge labeling on $G$; the edge $uv$ receives the label $f (u) + f (v)$. A graph $G$ is $A$-cordial if there is a vertex-labeling such that (1) the vertex label classes
differ in size by at most one and (2) the induced edge label classes differ in size by at most one.

Patrias and Pechenik studied the larger class of finite abelian groups $A$ such that all path
graphs are $A$-cordial. They posed a conjecture that all but finitely many paths graphs are $A$-cordial for any Abelian group $A$. In this paper we solve this conjecture. Moreover we show that all cycle graphs are $A$-cordial for any Abelian group $A$ of odd order.
\end{abstract}


\section{Introduction}
Assume $A$ is a finite Abelian group of order $n$ with the operation denoted by $+$.  For convenience we will write $ka$ to denote $a + a + \ldots + a$ (where the element $a$ appears $k$ times), $-a$ to denote the inverse of $a$, and use $a - b$ instead of $a+(-b)$.  Moreover, the notation $\sum_{a\in S}{a}$ will be used as a short form for $a_1+a_2+a_3+\dots+a_t$, where $a_1, a_2, a_3, \dots,a_t$ are all elements of the set $S$. The identity element of $A$ will be denoted by $0$.

For a graph $G=(V,E)$,  an Abelian group $A$ and an $A$-labeling
$c: V \to A$ let $v_c(a)=|c^{-1}(a)|$.
 The labeling $c$ induces an edge
labeling $c^*:E \to A$  defined by $c^*(e)=\sum_{v \in
e}c(v)$, let $e_{c}(a)=|{c^*}^{-1}(a)|$. The labeling $c$ is called an $A$-\textit{cordial labeling} if
$|v_c(a)-v_c(b)| \leq 1$ and $|e_{c}(a)-e_{c}(b)|\leq 1$ for any
$a,b \in A$.
 A graph $G$ is said to be {\it $A$-cordial} if it admits an $A$-cordial
labeling $c$.

Cordial labeling of graphs was introduced by Cahit \cite{Cahit} as a weakened
version of graceful labeling and harmonious labeling. This notion was generalized by Hovey for any Abelian group of order $k$ \cite{Hovey}.
So far research on $A$-cordiality has mostly focused on the case where $A$ is cyclic  and so called $k$-\textit{cordial}. Hovey \cite{Hovey} proved that all paths are $k$-cordial for any $k$ and raised the conjecture (still open for $k > 6$) that if $T$ is a tree, then it is $k$-cordial for every $k$. Also an $A$-cordial labeling of hypertrees was considered (see \cite{CGT,CGT1}).

Hovey proved that cycles are $k$-cordial for all odd $k$; for $k$ even $C_{2mk+j}$ is $k$-cordial when $0\leq j \leq \frac{k}{2}+ 2$ and when $k < j <2k$. Moreover he showed that $C_{(2m+1)k}$ is not $k$-cordial and posed a conjecture that  for $k$ even the cycle $C_{2mk + j}$ where $0 \leq j < 2k$, is $k$-cordial if and
only if $j \neq k$. This conjecture was verified by Tao \cite{Tao}.  Tao’s result combined with those of Hovey show that:
\begin{theorem}[\cite{Hovey,Tao}]\label{cycle}
The cycle $C_n$ is $k$-cordial if and only if $k$ is odd or  $n\neq 2mk+k$ for some positive integer  $m$.
\end{theorem}

Recently Patrias and Pechenik considered a dual problem \cite{Pechenik2}. Namely, let $\mathbb{G}$ be a family of graphs. We say a group $A$ is $\Ge$-\textit{cordial} if every $G \in \Ge$ is $A$-cordial. We say $A$ is \textit{weakly} $\Ge$-\textit{cordial} if all but finitely many $G \in \Ge$ are $A$-cordial.

Let $\Pe$ denote the class of path graphs and $\Ce$ denote the class of cycle graphs. Patrias and Pechenik \cite{Pechenik2} posed the following conjecture.
\begin{conjecture}[\cite{Pechenik2}]\label{Pechenik}
All finite Abelian groups are weakly $\Pe$-cordial.
\end{conjecture}
They proved the following:
\begin{theorem}[\cite{Pechenik2}]\label{not}
If $A\cong\left(\zet_2\right)^m$, then $P_{2^m}$ and $P_{2^m+1}$ are
not $A$-cordial (and so $A$ is not $\Pe$-cordial).
\end{theorem}
\begin{theorem}[\cite{Pechenik2}]\label{en}
Suppose $|A| = n$. Then $A$ is $\Pe$-cordial if and only if $P_n$ is $A$-cordial.
\end{theorem}

This problem is strongly connected with the concept of \textit{harmonious group} defined by Beals et al. \cite{Beals}.
We say that a finite group (not necessary Abelian) $A$ is \textit{harmonious} if the elements of $G$ can
be listed $g_1, g_2,\ldots, g_n$, so that $A= \{g_1g_2, g_2g_3,\ldots, g_ng_1\}$.
Analogously, letting $A^{\#}=\{g_1', g_2',\ldots, g_{n-1}'\}$  denote the set of non-identity elements of $A$, we
say  is harmonious if there is a listing $g_1', g_2',\ldots, g_{n'-1}$ of the elements of
$A^{\#}$ such that $A^{\#} = \{g_1'g_2', g_2'g_3',\ldots, g_{n-1}'g_1'\}$.
Note that if an Abelian group $A$ is harmoniuos then  a cycle $C_n$ is $A$-cordial and by Theorem~\ref{en} $A$ is $\Pe$-cordial.

\begin{theorem}[\cite{Beals}]\label{Beals}
If $A$ is a finite, non-trivial Abelian group, then $A$ is
harmonious if and only $A$ has a non-cyclic or trivial Sylow $2$-subgroup
and $A$ is not an elementary 2-group. Moreover, if $A$ has either a non-cyclic or trivial
Sylow $2$-subgroup, then $A^{\#}$ is harmonious, unless $A\cong\zet_3$.
\end{theorem}

The older concept is $R$-sequenceability of groups. A group $A$ of order $n$ is said to be $R$-\textit{sequenceable} if the nonidentity elements of the group can be listed in a sequence $g_1, g_2,\ldots, g_{n-1}$ such that  $g_1^{-1}g_2, g_2^{-1}g_3,\ldots, g_{n-1}^{-1}g_1$ are all distinct. This concept was introduced in  1974 by Ringel \cite{Ringel}, who used this concept in his solution of Heawood map coloring problem.
 An abelian group is $R^*$-sequenceable if it has an $R$-sequencing $g_1, g_2,\ldots, g_{n-1}$ such that $g_{i-1}g_{i+1}=g_i$ for some $i$ (subscripts are read modulo $n - 1$). The term was introduced by Friedlander et al. \cite{Friedlander}, who showed that the existence of an $R^*$-sequenceable Sylow 2-subgroup is a sufficient condition for a group to be $R$-sequenceable. It was proved the following:
\begin{theorem}[\cite{Headley}]\label{Headley} An abelian group whose Sylow 2-subgroup is noncyclic and not of order 8 is $R^*$-sequenceable.
\end{theorem}

Another corresponding problem is a concept of sum-rainbow Hamiltonian cycles on Abelian groups \cite{Lev}. Given a finite Abelian group $A$, consider the complete graph on the set of all elements of
$A$. Find a Hamiltonian cycle in this graph and for each pair of consecutive vertices along
the cycle compute their sum. What are the smallest $\sigma_{\min}(A)$ and the largest $\sigma_{\max}(A)$ possible number of
distinct sums that can emerge in this way. Recall that any group element $\iota\in A$ of order 2 (i.e., $|\iota|=2$) is called an \emph{involution}. 
\begin{theorem}[\cite{Lev}]\label{Lev}
For any finite non-trivial abelian group $A$ we have
$$\sigma_{\max}(A)=\begin{cases}
|A|& if\;A\; is\; neither\; a\; one\mathrm{-}involution\; group,\\
& nor\; an\; elementary\; Abelian \;2\mathrm{-}group;\\
|A|-1& if \; A \;  is\; a \;one\mathrm{-}involution\; group;\\
|A|-2& if \; A \;  is\; an\; elementary\; Abelian \;2\mathrm{-}group\; and\; |A|>2;
\end{cases}$$
\end{theorem}

 In this paper we will proof that all finite Abelian groups are weakly $\Pe$-cordial. Moreover we show that all finite Abelian groups of odd order are  $\Ce$-cordial. Note that the second result implies that in  the complete graph on the set of all elements of an Abelian group
$A$ of odd order, there exists a rainbow cycle $C_r$ for any $r=3,\ldots,|A|$. In the last section we show cordial labelings for $2$-regular graphs in some Abelian groups. 
\section{All finite Abelian groups are weakly $\Pe$-cordial}

In this section we prove that Conjecture~\ref{Pechenik} is true. We will use the following lemma.

\begin{lemma}[\cite{Pechenik2}] Suppose $|A|= n$ and let $k,m$ be positive integers. If $P_k$ and $P_{mn}$ are both
$A$-cordial, then so is $P_{mn+k}$.\label{extension}
\end{lemma}
\begin{theorem}\label{paths}
All finite Abelian groups are weakly $\Pe$-cordial.
\end{theorem}
\begin{proof}

Theorems~\ref{not}, \ref{en} and~\ref{Lev} imply that a finite Abelian group $A$ is $\Pe$-cordial if and only if it is not a nontrivial product
of copies of $\zet_2$ (equivalently, if and only if there exists $a\in A$ with $|a| > 2$).

Assume that $A\cong (\zet_2)^p$ for some $p$. For $p=2$ it was proved in (\cite{Pechenik}, Theorem 3.4), whereas for $p=3$ in (\cite{Pechenik2}, Proposition 4.1.). Therefore we consider the case $p>3$. 
Let $k=2^p$ and $a_1,a_2,\ldots,a_{k-1}$ be an $R^*$-sequence of $A$ which exists by Theorem~\ref{Headley}. Note that since $a_i=-a_i$ we obtain that
 the sequence $a_1+a_2,a_2+a_3,\ldots,a_{k-1}+a_1$ is injective. Therefore any path $P_m$ for $m<k$ is $A$-cordial. Without loss of generality we can assume that $a_2=a_1+a_3$. Note that this implies $a_1+a_2=a_3$ and $a_2+a_3=a_1$.

 We will show that for $k+1\leq m\leq 3k$ there exists an $A$-cordial labeling of $P_m$.\\
Take the sequence of length $2k$:
$$a_2,a_3,\ldots,a_{k-1},a_1,0,0,a_3,a_4,\ldots,a_{k-1},a_1,a_2.$$
When we take first $m>k$ elements of this sequence, then we obtain an $A$-cordial labeling of $P_m$.\\
Take the sequence of length $3k$:
$$a_1,a_2,\ldots,a_{k-1},a_1,0,0,0,a_3,a_4,\ldots,a_{k-1},a_1,a_2,a_2,a_3,\ldots,a_{k-1}.$$
When we take first $m>2k$ elements of this sequence, then we obtain an $A$-cordial labeling of $P_m$.
Using Lemma~\ref{extension} we obtain that $P_m$ is $A$-cordial for any $m\geq 3k+1$.

\end{proof}

\section{The main result}

We will start with this useful lemma.

\begin{lemma} If $k=|A|$ is odd, then $C_{mk+r}$ for $m>0$ is $A$-cordial.\label{long}
\end{lemma}
\begin{proof}
By Theorem~\ref{Lev} there exists an $A$-cordial labeling $c_1$ of $C_k=v_1,\ldots,v_k$. We can ssume that $0<r\leq k-1$. Let
$$C_{mk+r}=v_{1,1},\ldots,v_{1,m},x_{1},v_{r,1},\ldots v_{r,m},x_{r},v_{r+1,1},\ldots,v_{r,m}.\ldots v_{k,1},\ldots,v_{k,m}.$$
Set $c_2(v_{i,j})=c_1(v_i)$ for $i=1,\ldots,k$ and $j=1,\ldots,m$ and $c_2(x_{i})=c_1(v_{i})$ for $i=1,\ldots,r$. 

Observe that 
$$c_2^*(x,y)=\begin{cases}
2c_1(v_i)&\mathrm{for}\; x=v_{i,j},y=v_{i,j+1},\;i=1,\ldots,k,\;j=1,\ldots,m-1,\\
2c_1(v_i)&\mathrm{for}\; x=v_{i,m},y=x_i,\;i=1,\ldots,r,\\
c_1(v_i)+c_1(v_{i+1})&\mathrm{for}\; x=x_i,y=v_{i+1,1},\;i=1,\ldots,r,\\
c_1(v_i)+c_1(v_{i+1})&\mathrm{for}\; x=v_{i,m},y=v_{i+1,1},\;i=r+1,\ldots,k.
\end{cases}$$

Note that since $|A|$ is odd a function $h \colon A\to A$ defined as $h(g)=2g$ is an authomorphism. Thus $e_{c_2}(2c_1(v_i))=m+1$ for $i=1,\ldots,r$ 
and $e_{c_2}(2c_1(v_i))=m$ for $i=r+1,\ldots,k$

Obviously $c_2$ is an $A$-cordial labeling of $C_{mk+r}$.

\end{proof}

\begin{lemma}\label{koniec}Suppose $|A|=k=2l+1$  and $3\leq r \leq k$, then there exists such $k$-cordial labeling $c$ of cycle $C_{r}=v_1\ldots v_r$ such that  is $c(v_1)=0$ and $c(v_r)=l+1$.
\end{lemma}
\begin{proof}
Suppose first that $r$ is even by result of Hovey (\cite{Hovey}, Theorem 9) there exists such $k$-cordial labeling $c'$ of $C_r$ that $c'(v_1)=0$, $c'(v_r)=1$.
Since $\gcd(2l+1,l+1)=1$ the labeling defined as $c(v)=(l+1)c'(v)$ for any $v\in V(C_r)$ is  a $k$-cordial labeling.\\

For $r=2r'+1$  we will use the idea of sequential labeling  introduced by Grace \cite{Grace}.
Let $c(v_{2i+1})=i$ for $i=0,1,\ldots,r'$ and $c(v_{2i})=l+i$ for $i=1,2,\ldots,r'$. One can easily check that this is a  $k$-cordial labeling of $C_r$. 
\end{proof}

\begin{lemma} Let $h$, $m$ and $l\leq h$  be  odd positive integers, $3\leq k\leq h$ be an integer.  Suppose that  $n-k=r_1+r_2+\ldots+r_{l}$ for  
\begin{itemize}
	\item $2\leq r_i \leq m-1$ for $i\in\{1,2\ldots,l\}$ or
	\item $2\leq r_i \leq m-1$ for $i\in\{1,2\ldots,l-2\}$, $1\leq r_{l-1}\leq m-2$ and $r_l=1$.	
\end{itemize}

There exists an $H\times \zet_m$-cordial labeling of $C_n$ for any Abelian group $H$ of order $h$ such that $C_k$ is $H$-cordial.
\label{usefull}\end{lemma}
\begin{proof}
By the assumption there exists an injective $H$-cordial labeling $c'$ of $C_k=v_1,\ldots,v_k$. \\\\
Since $n=\sum_{i=1}^{l}r_i+k$ we can define $C_n$ as 
$$C_{n}=v_{1,1}\ldots,v_{1,r_1+1},v_{2,1}\ldots,v_{2,r_2+1},\ldots,v_{l,1}\ldots v_{l,r_{l}+1},v_{l+1,1},v_{l+2,1}\ldots v_{k,1}.$$

\textit{Case 1.} $2\leq r_i \leq m-1$ for $i\in\{1,2\ldots,l\}$.\\
By Lemma~\ref{koniec} there exist  $\zet_m$-cordial labelings $c_i$ of $C_{r_i+1}=x_1,\ldots x_{r_i+1}$ that $c_i(x_1)=0$ and $c_i(x_{r_i+1})=\left\lceil m/2\right\rceil$  for $i\in\{1,2\ldots,l\}$.

Set $$c(v_{i,j})=\begin{cases}(c'(v_i),c_i(x_j))& \mathrm{for}\;\; i=1,\ldots,l,\; j=1,\ldots,r_i+1,\\
(c'(v_i),0)& \mathrm{for}\;\; i=l+1,l+3,\ldots,k-1,\\
\left(c'(v_i),\left\lceil \frac{m}{2}\right\rceil\right)& \mathrm{for}\;\; i=l+2,l+4,\ldots,k.\\
\end{cases}$$
Observe that  
$$c^*(xy)=\begin{cases}
(2c'(v_i),c_i(x_j)+c_i(x_{j+1}))&\mathrm{for}\; x=v_{i,j},y=v_{i,j+1},\;i=1,\ldots,l,\;j=1,\ldots,r_i,\\
\left(c'(v_i)+c'(v_{i+1}),\left\lceil \frac{m}{2}\right\rceil\right)&\mathrm{for}\; x=v_{i,r_i+1},y=v_{i+1,1},\;i=1,\ldots,l,\\
&\mathrm{or}\; x=v_{i,1},y=v_{i+1,1},\;i=l+1,\ldots,k.
\end{cases}$$
Recall that for $|H|$  odd a function $h \colon H\to H$ defined as $h(g)=2g$ is an authomorphism. Moreover $c_i(x_j)+c_i(x_{j+1})\neq \left\lceil \frac{m}{2}\right\rceil$ for $i=1,\ldots,l,$ $j=1,\ldots,r_i$, therefore $c$ is an $H\times \zet_m$-cordial labeling for $C_n$.\\

\textit{Case 2.} $2\leq r_i \leq m-1$ for $i\in\{1,2\ldots,l-2\}$, $1\leq r_{l-1}\leq m-2$ and $r_l=1$.\\
By Lemma~\ref{koniec} there exist  $\zet_m$-cordial labelings $c_i$ of $C_{r_i+1}=x_1,\ldots x_{r_i+1}$ that $c_i(x_1)=0$ and $c_i(x_{r_i+1})=\left\lceil m/2\right\rceil$  for $i\in\{1,2\ldots,l-2\}$. By Lemma~\ref{koniec} there exists a  $\zet_m$-cordial labeling $c_{l-1}$ of $C_{r_{l-1}+2}=x_1,\ldots x_{r_{l-1}+2}$ that $c_i(x_1)=0$ and $c_i(x_{r_{l-1}+2})=\left\lceil m/2\right\rceil$.

Set $$c(v_{i,j})=\begin{cases}(c'(v_i),c_i(x_j))& \mathrm{for}\;\; i=1,\ldots,l-1,\; j=1,\ldots,r_i+1,\\
(c'(v_{l}),\left\lceil m/2\right\rceil-c_{l-1}(x_{r_{l-1}+1}))& \mathrm{for}\;\;i=l,\;j=1,\\ 
(c'(v_{l}),\left\lceil m/2\right\rceil)& \mathrm{for}\;\;i=l,\;j=2,\\ 
(c'(v_i),0)& \mathrm{for}\;\; i=l+1,l+3,\ldots,k-1,\\
\left(c'(v_i),\left\lceil \frac{m}{2}\right\rceil\right)& \mathrm{for}\;\; i=l+2,l+4,\ldots,k.\\
\end{cases}$$
Note that $ c_{l-1}(x_{r_{l-1}+1})\notin\{0,\left\lceil m/2\right\rceil\}$, thus $c$ is injective and $c^*(v_{l,1}v_{l,2})\neq (2c'(v_l),\left\lceil m/2\right\rceil,0)$. Therefore applying the same arguments as above we are done.

\end{proof}

\begin{lemma} If $|A|=k$ is odd, then $C_r$ for $3\leq r< k$ is $A$-cordial.\label{short}
\end{lemma}
\begin{proof}
The Fundamental Theorem of Finite Abelian Groups states that a finite Abelian
group $A$ of order $k$ can be expressed as the direct product of cyclic subgroups of prime-power order. This implies that
$$
	\Gamma\cong\zet_{p_1^{\alpha_1}}\times\zet_{p_2^{\alpha_2}}\times\ldots\times\zet_{p_t^{\alpha_t}}\;\;\; \mathrm{where}\;\;\; k = p_1^{\alpha_1}\cdot p_2^{\alpha_2}\cdot\ldots\cdot p_t^{\alpha_t}
$$
and $p_i$ for $i \in \{1, 2,\ldots,t\}$ are primes, not necessarily distinct. This product is unique up to the order of the direct product.

Without loss of generality we can assume that $p_1^{\alpha_1}\geq p_2^{\alpha_2}\geq \ldots \geq p_t^{\alpha_t}$.
If  $p_1^{\alpha_1}\geq r$ then by Theorem~\ref{cycle} there exists a $p_1^{\alpha_1}$-cordial  labeling $c'$ of $C_r$. Define $c$ as $c(v_i)=(c'(v_i),0)\in A$ which is obviously an $A$-cordial labeling of $C_r$. 

Let now $j\geq2$ be  such index that $e=p_1^{\alpha_1}\cdot\ldots\cdot p_j^{\alpha_j}\leq r$ and $p_1^{\alpha_1}\cdot\ldots\cdot p_{j+1}^{\alpha_{j+1}}> r$.
Then for $H=\zet_{p_1^{\alpha_1}}\times\ldots\times\zet_{p_j^{\alpha_j}}$, $B=\zet_{p_{j+2}^{\alpha_{j+2}}}\times\ldots\times\zet_{p_t^{\alpha_t}}$ and $m=p_{j+1}^{\alpha_{j+1}}$ the group $A\cong H\times\zet_m\times B$. Note that $e\geq m$.


Observe that $r-e=\beta\cdot(m-1)+r'$ for $0\leq r' <m-1$ and $0\leq\beta<e$. If $\beta=0$ and $r'=0$, then by Theorem~\ref{Lev} there exists an $H$-cordial labeling $c_1$ of $C_e=v_1,\ldots,v_e$. Hence  the labeling $c$ defined as $c(v)=(c_1(v),0,0)$ is $A$-cordial. Therefore we can assume that $r'>0$ or $\beta>0$. We will consider two cases on $m$.\\

\textit{Case 1.} $m=3$

Observe that $r'\in\{0,1\}$. Suppose first that $r'=0$. By Theorem~\ref{Beals} there exists an $H$-cordial labeling of $C_e$. If $\beta$ is odd then $r-e=r_1+r_2+\ldots+r_{\beta}$ such that  $ r_i =2$ for  $i\in\{1,2,\ldots,\beta\}$.  By Lemma~\ref{usefull} there exists an $H\times \zet_3$-cordial labeling $c'$ of $C_r$, thus the labeling $c$ defined as $c(v)=(c'(v),0)$ is $A$-cordial. If $\beta>0$ is even, then $r-e=r_1+r_2+\ldots+r_{\beta+1}$ such that  $ r_i =2$ for  $i\in\{1,2,\ldots,\beta-1\}$ and $r_{\beta}=r_{\beta+1}=1$ and as before we apply Lemma~\ref{usefull}.

Suppose now that $r'=1$. If $e>3$ then by Theorem~\ref{Beals} there exists an $H$-cordial labeling (one can say an $H^{\#}$-cordial) of $C_{e-1}$. Note that  $r-(e-1)=r_1+r_2+\ldots+r_{\beta+1}$ such that  $ r_i =2$ for  $i\in\{1,2\ldots,\beta+1\}$ as before we apply Lemma~\ref{usefull}.

Observe that for $e=3$ there is $r\in\{4,6,8\}$. The labeling is presented in Figure~\ref{z33}.
\begin{figure}
\begin{center}
\includegraphics[width=12cm]{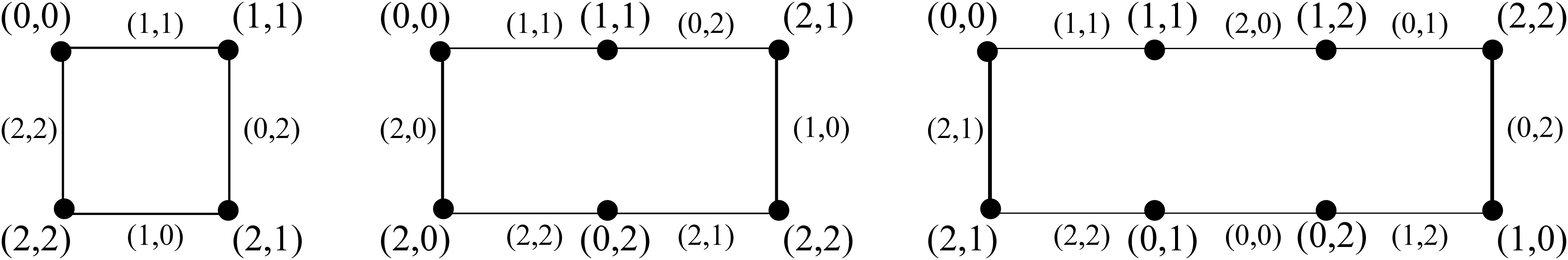}\label{z33}
\caption{A $\zet_3\times\zet_3$-cordial labeling for $C_{4}$, $C_6$ and $C_8$.}
\end{center}
\end{figure}

\textit{Case 2.} $m>3$

If $\beta=0$, then $r'\geq 1$ and  $e>3$. By Theorem~\ref{Beals} there exists an $H$-cordial labeling  of $C_{e-1}$. Note that  $r-(e-1)=r_1$ such that  $ 2\leq r_1 \leq m-1$ and as above we apply Lemma~\ref{usefull}.

Assume now that   $\beta>0$ or $r'>0$, by Theorem~\ref{Beals} there exists an $H$-cordial labeling of $C_e$. 
Let $r-e=r_1+r_2+\ldots+r_{l}$. \\
If  $\beta$ is even $r'\geq 2$ then $r_i=m-1$ for $i\in\{1,2\ldots,\beta\}$ and $r_{\beta+1}=r'$.\\
If  $\beta\geq 2$ even and $r'< 2$ then $r_i=m-1$ for $i\in\{1,2\ldots,\beta-1\}$, $r_{\beta}=m-3\geq2$, $r_{\beta+1}=2+r'\leq 3$.\\
If  $\beta$ is odd and $r'= 0$ then $r_i=m-1$ for $i\in\{1,2\ldots,\beta\}$.\\
If  $\beta$ is odd and $r'\geq 1$ then $r_i=m-1$ for $i\in\{1,2\ldots,\beta-1\}$, $r_{\beta}=m-2\geq2$, $r_{\beta+1}=r'\leq m-2$, $r_{\beta+2}=1$.\\

By Lemma~\ref{usefull} there exists an $H\times \zet_m$-cordial labeling $c'$ of $C_r$, thus the labeling $c$ defined as $c(v)=(c'(v),0)$ is $A$-cordial.

\end{proof}

\begin{figure}
\begin{center}
\includegraphics[width=12cm]{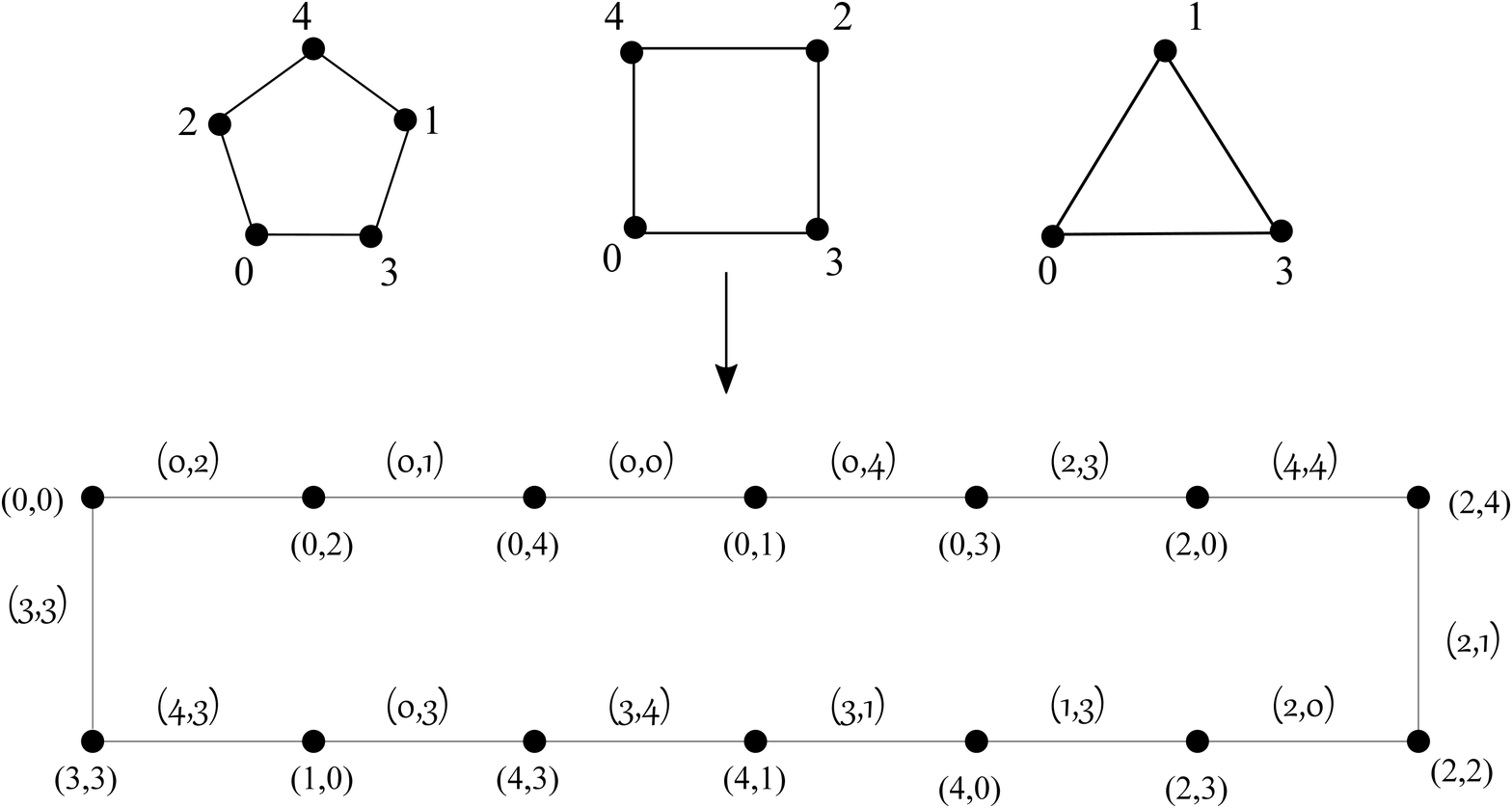}\label{c12}
\caption{A $\zet_5\times\zet_5$-cordial labeling for $C_{14}$.}
\end{center}
\end{figure}

Lemmas ~\ref{long} and \ref{short} implies the following.
\begin{theorem} If $|A|$ is odd, then $A$ is $\Ce$-cordial.
\end{theorem}

\section{$2$-regular graphs}

In this section we show cordial labelings for $2$-regular graphs for some Abelian groups. 
Let Bij$(A)$ denote the set of all bijections from $A$ to itself.

A \textit{complete mapping} of a group $A$ is defined
as $\varphi\in$Bij$(A)$ that the mapping  $\theta\colon g \mapsto  g^{-1}\varphi(g)$ is
also bijective \cite{Friedlander2,Hall}. (Some authors refer to $\theta$, rather than $\varphi$, as the complete
mapping.)  Thus $A$  an $R$-sequenceable  if
and only if it has a complete mapping which fixes the identity elements and permute the remaining elements cyclically. 
Complete mappings have  been studied since 1944 \cite{Mann}, initially for their connection to sets of
mutually orthogonal Latin squares. For finite Abelian groups it was proved the following:
\begin{theorem}[\cite{Hall,Hall2}]\label{bijection}
A finite Abelian $A$ group has a complete
mapping if and only if $I(A)\neq1$.
\end{theorem}
Observe  that $A$ is harmonious
if and only if $A$ has a complete mapping which is also a $|A|$-cycle.  We use a complete mapping to derive the following result.
\begin{theorem}\label{wazne}
Let $A$ and $B$  be  Abelian groups of order $m$ and $n$, respectively such that $I(B)\neq 1$. If $C_{m}$ is $A$-cordial, then $nC_{m}$ is $A \times B$
-cordial.
\end{theorem}
\begin{proof} 
Let $C_{m}^j=v_{1,j}v_{2,j}\ldots v_{m,j}$  for $j=1,\ldots,n$ and $G=C_m^1\cup\ldots\cup C_m^n=nC_m$. Let $B=\{b_1,\ldots,b_n\}$ and let $\varphi\in$Bij$(B)$ be a complete mapping that the mapping  $\theta\colon g \mapsto  -g+\varphi(g)$ is
also bijective, which exists by Theorem~\ref{bijection}.

Let $c_1$ be an $A$-cordial labeling of $C_{m}=x_1\ldots x_{m}$. 
Let $c\colon A\times B\to V(B)$ be such that 

$$c(v_{i,j})=\begin{cases}
(c_{1}(x_i),-\theta(b_j)),&i=1,3,\ldots,2\left\lceil \frac{m}{2}\right\rceil-1,j=1,2,\ldots,n\\
(c_{1}(x_i),\varphi(b_j)),&i=2,4,\ldots,2\left\lceil \frac{m}{2}\right\rceil,j=1,2,\ldots,n.
\end{cases}$$
 Obviously $v_c(z)=1$ for any $z\in A \times B$. 
Since $c^*(v_{i,j}v_{i+1,j})=(c^*_1(x_ix_{i+1}),b_j)$ we obtain that $c$ is an $A \times B$-cordial for $G$.
\end{proof}
If we make stronger assumption on groups $A$ and $B$ then we get the following.
\begin{theorem}\label{wazne2}
Let $A$ and $B$  be  Abelian groups of order $m$ and $n$, respectively such that $I(A),I(B)>1$. If $C_{m}$ is $A$-cordial and $C_n$ is $B$-cordial, then $2C_{mn/2}$ is $A \times B$-cordial.
\end{theorem}
\begin{proof} 
The condition $I(A)>1$ implies $|A|$  even. 
Let $V(2C_{mn/2})=\{v_{i,j}:i=1,\ldots,m,\;j=1,\ldots,n\}$ and $v_{i,j}v_{k,l}\in E(C_{mn})$ if  ($j=l$ and $k=i+1$) or ($l=j+2$, $i=m$ and $k=1$).   
Let $c_1$ be an $A$-cordial labeling of $C_{m}=x_1\ldots x_{m}$ and $c_2$ be a $B$-cordial labeling of $C_{n}=y_1\ldots y_{n}$. 
Let $c\colon A\times B\to V(G)$ be such that 

$$c(v_{i,j})=\begin{cases}
(c_{1}(x_i),c_{2}(y_j)),&i=1,3,\ldots,m-1,\;j=1,2,\ldots,n\\
(c_{1}(x_i),c_{2}(y_{j+1})),&i=2,4,\ldots,m,\;j=1,2,\ldots,n.
\end{cases}$$

Observe that $c^*(v_{i,j}v_{i+1,j})=(c^*_1(x_ix_{i+1}),c^*_2(y_jy_{j+1}))$ and $c^*(v_{m,j}v_{1,j+2})=(c_1(x_m)+c_1(x_{1}),c^*_2(y_{j+1}y_{j+2}))$.
We obtain that $e_c((a,b))=1$ for any $(a,b)\in A\times B$.
\end{proof}
\section*{Acknowledgement}
The work of the author was partially supported by the Faculty of Applied Mathematics AGH UST statutory tasks within subsidy of Ministry of Science and Higher Education.
The author would like to thank prof. Agnieszka G\"{or}lich for her valuable comments.


\end{document}